\documentclass[a4paper,11pt]{amsart}
\usepackage[margin=3cm]{geometry}

\usepackage[T1]{fontenc}
\usepackage{lmodern, amsfonts,amsmath,amstext,amsbsy,amssymb,
amsopn,amsthm,upref,eucal,mathptmx,mathtools,url}

\usepackage{mathrsfs}

\RequirePackage{xcolor} 
\definecolor{halfgray}
{gray}{0.55}
\definecolor{webgreen}
{rgb}{0,0.4,0}
\definecolor{webbrown}
{rgb}{.8,0.1,0.1}
\definecolor{red}
{rgb}{1,0,0}
\usepackage{microtype}
\usepackage{tikz}

\newcommand \R {{ \mathbb R}}

\newcommand*{\diff}{\mathop{}\!\mathrm{d}}

\newcommand{\SL}{%
\operatorname{SL}
}
\newcommand{\PSL}{%
\operatorname{PSL}
}

\DeclareMathOperator{\Leb}{Leb}

\newtheorem{theorem}{Theorem}
\newtheorem {lemma}[theorem]{Lemma}
\newtheorem {proposition}[theorem]{Proposition}

\newtheorem{definition}[theorem]{Definition}

\date{\today}

\author{Davide Ravotti}

\address{
Monash University, School of Mathematics \\ Clayton Campus, 3800 Victoria, Australia
}

\email{davide.ravotti@gmail.com\\}

 \title[Polynomial mixing for time-changes of unipotent flows]
 {Polynomial mixing for time-changes of unipotent flows}

\begin{document}
\maketitle

\begin{abstract}
Let $G$ be a connected semisimple Lie group with finite centre, and let $M= \Gamma \backslash G$ be a compact homogeneous manifold. Under a spectral gap assumption, we show that smooth time-changes of any unipotent flow on $M$ have polynomial decay of correlations. Our result applies also in the case where $M$ is a finite volume, non-compact quotient under some additional assumptions on the generator of the time-change. This generalizes a result by Forni and Ulcigrai (JMD, 2012) for smooth time-changes of horocycle flows on compact surfaces.
\end{abstract}

\section{Introduction}

In dynamical systems, the term \emph{parabolic} is often used to refer to (for us, continuous-time) systems for which nearby orbits diverge \lq\lq slowly\rq\rq.
Smooth parabolic flows have an intermediate chaotic behaviour: on one hand, they tend to exhibit properties typical of strongly chaotic systems, such as mixing and absolutely continuous spectrum, on the other, they have zero topological and metric entropy.
Classical examples of parabolic systems in homogeneous dynamics are unipotent flows on semisimple Lie groups and nilflows on nilmanifolds. Smooth area-preserving flows on higher genus surfaces can also be classified as (non-uniformly) parabolic.
We refer the reader to \cite[Chapter 8]{HasKat} for an extensive discussion on parabolic phenomena.

In part due to the lack of a unified theory of parabolic dynamics, and towards a better understanding of its common features, there has been an increasing interest in studying new non-homogeneous parabolic systems. 
Important sources of examples are smooth perturbations of homogeneous flows. 
Perhaps the simplest type of perturbation is given by performing a \emph{time-change}, or a \emph{time-reparametrization}: leaving the orbits unchanged, one varies smoothly the speed of the points. 
Despite their apparent simplicity, time-changes can alter significantly the ergodic properties of the original flow.
For example, although nilflows are never weak-mixing, generic non-trivial time-changes of ergodic nilflows are mixing \cite{afu:heisenberg, fornikan, ravotti:nilflow, afru}. Different types of perturbations, including skew-product constructions \cite{simonelli:spectrum} and others \cite{ravotti:sl3}, have been investigated as well.

In this paper, we focus on smooth time-changes of unipotent flows.
The case of horocycle flows is better understood than the general case. 
A classical result due to Marcus \cite{marcus:horocycle}, generalizing a previous work by Kushnirenko \cite{Kus}, shows that smooth time-changes on compact surfaces are mixing. More recently, Tiedra de Aldecoa \cite{tiedra:spectrum} showed that they have absolutely continuous spectrum, and, independently and at the same time, Forni and Ulcigrai \cite{forniulcigrai:timechanges} proved that the spectrum is Lebesgue. This result on the spectral type has been generalized by Simonelli to smooth time-changes of ergodic unipotent flows \cite{simonelli:spectrum}.
Finer properties, including the countable multiplicity of the spectrum and other remarkable rigidity results, have been investigated by several authors \cite{ratner2, ratner3, ffk, KLU, ffmob}. 

Much less is known about their quantitative properties. 
In the case of nilflows, Forni and Kanigowski proved polynomial mixing for generic time-changes of a full-dimensional set of Heisenberg nilflows \cite{fornikan}. This is the only quantitative mixing result for time-changes of nilflows.
Going back to the horocycle flow, Forni and Ulcigrai showed that the mixing rate of smooth time-changes on compact surfaces is polynomial \cite{forniulcigrai:timechanges}. More precisely, they proved that the correlations of smooth observables can be bounded by the rate of equidistribution of sheared geodesic segments. Building on their result, Kanigowski and the author showed that the rate of 3-mixing is also polynomial \cite{ravottikanig}.
It is currently not known whether the estimates in \cite{forniulcigrai:timechanges} are optimal.

In this paper, we generalize the result by Forni and Ulcigrai to time-changes of general unipotent flows on finite volume manifold, which are not necessarily compact.
We summarize our result as follows, for a precise statement see Theorem~\ref{thm:main}.

\begin{theorem}\label{thm:main_0}
Let $G$ be a connected semisimple Lie group with finite centre and no compact factor, and let $\Gamma < G$ be an irreducible lattice. Let $\{ h_t\}_{t \in \R}$ be a unipotent flow on $M = \Gamma \backslash G$, equipped with the normalized Haar measure $\mu$. Let $\tau \colon M \to \R_{>0}$ be a positive smooth function, and let $\{h^\tau_t\}_{t \in \R}$ be the time-change induced by $\tau$. If the time-change is admissible (see Definition \ref{def:admiss}), then there exists $0<\alpha<1$ and, for all $f, g \in \mathscr{C}^{\infty}_c(M)$, there exists a constant $C_{f,g} \geq 0$ such that for all $t \geq 1$ we have
$$
\left\lvert \int_M f \circ h^{\tau} \cdot g \diff \mu^{\tau} - \mu^{\tau}(f) \mu^{\tau}(g) \right\rvert \leq C_{f,g} t^{-\alpha}.
$$
\end{theorem}

The admissibility condition in the statement of Theorem \ref{thm:main_0} is introduced in Definition \ref{def:admiss}. We point out that, in the case of a compact space $M$, any positive smooth function induces an admissible time-change. 
When the space is non-compact, however, we need to impose some non-degeneracy condition on the behaviour of the generator $\tau$ in the cusps. A similar assumption appears in the works of Kushnirenko \cite{Kus} and of Simonelli \cite{simonelli:spectrum}. 

Not surprisingly, the proof of Theorem \ref{thm:main_0} follows a \lq\lq mixing via shearing\rq\rq\ approach, analogous to the strategy employed by Forni and Ulcigrai in \cite{forniulcigrai:timechanges}, which dates back to Marcus \cite{marcus:horocycle}. The idea of studying mixing and other strong chaotic properties of smooth parabolic flows by analyzing the action on transverse arcs has been used successfully in several different settings.
In order to make it effective and prove quantitative results, one needs good control on the growth of ergodic integrals. The main difficulty in the setting of this paper is that no good renormalization is known for non-horospherical unipotent flows, so that no general pointwise estimate on the deviations of their ergodic averages is available. The idea is then to replace pointwise estimates with $L^2$-estimates, which are sufficient for mixing. More precisely, we exploit pointwise polynomial bounds on set of polynomially small measure, see Proposition \ref{thm:bound_erg_int}. In turn, the $L^2$-bounds on ergodic integrals are deduced easily from quantitative mixing estimates for the unipotent flow, which, under a spectral gap assumption, are well-known.


\section{Preliminaries and statement of the main result}

In this section, we recall some basic notions on unipotent flows and the Jacobson-Morozov Theorem. We state an important result on their quantitative mixing properties, which follows from the work several authors on effective decay of matrix coefficients for unitary representations of semisimple groups. We then introduce the notion of admissibility for smooth time-changes and in \S\ref{sec:main_result} we state our main result.

\subsection{Unipotent flows and the Jacobson-Morozov Theorem}Let $G$ be a semisimple Lie group with finite centre, and let $\Gamma < G$ be a lattice subgroup. Denote by $M $ the homogeneous space $ \Gamma \backslash G$. We remark again that we do not assume $M$ to be compact. 
The Haar measure on $G$ descends to a finite measure $\mu$ on $M$, which we will assume to be normalized to a probability measure. It is invariant by the right-action of $G$ on $M$.

Elements of the Lie algebra $\mathfrak{g}$ of $G$ are in one-to-one correspondence with 1-parameter subgroups of $G$, namely any 1-parameter subgroup of $G$ is of the form  $\{ \exp(tV) : t \in \R\}$ for some $V \in \mathfrak{g}$.
The homogeneous flow $\{\phi^V_t\}_{t \in \R}$ generated by $V \in \mathfrak{g} \setminus \{0\}$ is the smooth flow on $M$ defined by the restriction of the right-action of $G$ to the corresponding 1-parameter subgroup. Explicitly, it is given by
$$
\phi^V_t(\Gamma g) = \Gamma g \exp(tV).
$$
An element $U \in \mathfrak{g}\setminus \{0\}$ is \emph{unipotent}, and the corresponding $\{\phi^U_t\}_{t \in \R}$ is a \emph{unipotent flow}, if $\mathfrak{ad}_U = [U, \cdot]$ is a (non-zero) nilpotent linear operator on $\mathfrak{g}$.

Let us fix a unipotent flow $h_t = \phi_t^U$ induced by $U\in \mathfrak{g} \setminus \{0\}$.
The Jacobson-Morozov Theorem ensures the existence of a subalgebra of $\mathfrak{g}$ containing $U$ which is isomorphic to $\mathfrak{sl}_2(\R)$. In particular, there exists an element $X \in \mathfrak{g} \setminus \{0\}$ such that $[X,U] = -U$. 
By taking the exponential, we have the commutation relation
\begin{equation}\label{eq:commutation}
h_t \circ \phi^X_r (x) = \phi^X_r \circ h_{e^r t} (x)
\end{equation}
which holds for all $x \in M$ and $t,r \in \R$.
By analogy to $\SL(2,\R)$, we will sometimes call the homogeneous flow $\{\phi^X_t\}_{t \in \R}$ the \emph{geodesic flow}.


\subsection{Effective decay of matrix coefficients}

Denote by $H$ the Hilbert space $L^2(M)$ and let $H_0$ be the subspace of $H$ consisting of functions with zero average. 
We say that $M$ satisfies the \emph{strong spectral gap assumption} if the regular representation $\rho_0$ of $G$ on $H_0$ has a strong spectral gap; that is, the restriction of $\rho_0$ to any compact factor of $G$ is isolated from the trivial representation.
The strong spectral gap assumption is known to hold if, for example, $G$ is a semisimple group with finite centre and without compact factors and $\Gamma$ is an irreducible lattice \cite{KleMar, KelSar}, or when $G$ admits a simple factor of rank at least two which acts ergodically on $M$.

It is known from the work of Harish-Chandra \cite{HarCha}, Borel and Wallach \cite{BW}, Cowling \cite{Cow}, Howe \cite{How}, Moore \cite{Moo}, Katok and Spatzier \cite{KS}, and others, that the spectral gap condition provides explicit estimates on the decay of matrix coefficients for $\rho_0$. 
In the case of $\SL(2,\R)$, Ratner \cite{ratner2} established sharp bounds for general H{\" o}lder observables.
Theorem \ref{thm:exp_mix} below contains the bounds we need for our purposes, we refer the reader to the references mentioned above, as well as the work of Bj{\" o}rklund, Einsiedler, and Gorodnik \cite{BEG}, for precise effective statements on mixing and multiple mixing.
In the following, we will write $f \circ h_t$ for $\rho_0(\exp(tU))f$, and we will denote by $H^{\infty}$ (and by $H^{\infty}_0$) the subspace of $H$ (of $H_0$, respectively) of smooth vectors for the action of $G$.

\begin{theorem}\label{thm:exp_mix}
Assume that $M$ satisfies the strong spectral gap assumption. There exist $B, \beta >0$ and, for all $f,g \in H^{\infty}$ there exist $S(f), S(g) \geq 0$ (which depend on the $L^2$-norms of finitely many derivatives of $f$ and $g$ respectively) such that for all $t\geq 1$ we have
$$
\left\lvert \langle f \circ h_t, g \rangle - \mu(f) \mu(g)\right\rvert \leq B S(f)S(g) t^{-\beta}.
$$
\end{theorem}


\subsection{Time-changes}

Let $\tau \colon M \to \R_{>0}$ be a strictly positive smooth function.
The time-change of $\{h_t\}_{t \in \R}$ generated by $\tau$ is the smooth flow $\{h^{\tau}_t\}_{t \in \R}$ induced by the (non-homogeneous) vector field $\tau^{-1}U$. The orbits of $\{h^{\tau}_t\}_{t \in \R}$ are the same as the ones of $\{h_t\}_{t \in \R}$, but they are traveled at different speed.
Explicitly, for any $x \in M$ and $t \in \R$, let $u(x,t)$ be defined by the equality
\begin{equation}\label{eq:defin_u}
t= \int_0^{u(x,t)} \tau \circ h_s(x) \diff s.
\end{equation}
Then, $u(x,t)$ is an additive cocycle over the flow $\{h_t\}_{t \in \R}$; in other words, for all $t,r \in \R$, 
$$
u(x,t+r) = u(x,t) + u(h_t(x),r),
$$
and we have
$$
h^\tau_t(x) = h_{u(x,t)}(x).
$$
If $\mathscr{L}$ denotes the Lie derivative, it is easy to check that 
$$
\mathscr{L}_{\tau^{-1}U} ( \tau \diff \mu) = \mathscr{L}_{U} (\diff \mu) = 0,
$$
which implies that $\{h^{\tau}_t\}_{t \in \R}$ preserves the smooth measure $\mu^\tau$ equivalent to $\mu$ with density $\tau$. Without loss of generality, we will assume that $\mu(\tau)=1$, so that $\mu^\tau$ is a probability measure.

We now formulate the definition of admissibility, which constitutes our assumption on the generator $\tau$ of the time-change. 
\begin{definition}\label{def:admiss}
We will say that the time-change $\{h^\tau_t\}_{t \in \R}$ is \emph{admissible} if the generator $\tau \in H^{\infty}$ is a smooth vector and $\tau, \tau^{-1}, X\tau,$ and $X^2\tau$ are uniformly bounded. We define
$$
m_\tau := \max \{ \|\tau\|_\infty, \|\tau^{-1}\|_\infty, \|X\tau\|_\infty, \|X^2\tau\|_\infty \} \geq 1.
$$
\end{definition}
If $M$ is a compact space, then $\tau$ is admissible if and only if it is a smooth function. In the non-compact case, the admissibility condition provides some control on the behaviour of the generator in the cusps. As we already mentioned, a condition of similar nature was introduced by Kushnirenko in his work on mixing for time-changes of horocycle flows \cite{Kus} and appears also in the result by Simonelli \cite{simonelli:spectrum}. 

We conclude this section with the following observation, which is an immediate consequence of \eqref{eq:defin_u} and Definition \ref{def:admiss}.
\begin{lemma}\label{thm:ulessm}
If $\tau$ is admissible, for all $x \in M$ and $t \geq 0$, we have $m_\tau^{-1}t \leq u(x,t) \leq m_\tau t$.
\end{lemma}


\subsection{The main result}\label{sec:main_result}

We are now ready to state our main result. Roughly speaking, Theorem \ref{thm:main} says that whenever a unipotent flows has polynomial decay of correlations, then the same happens for any admissible time-change.

\begin{theorem}\label{thm:main}
Assume that $M$ satisfies the strong spectral gap assumption. Let $\{h^\tau_t\}_{t \in \R}$ be an admissible time-change of a unipotent flow $\{h_t\}_{t \in \R}$ on $M$. There exists $0 <\alpha < 1$ and a constant $C_\tau \geq 0$ such that for all $f,g \in \mathscr{C}^{\infty}_c(M)$, there exist ${\widetilde S}(f), {\widetilde S}(g) \geq 0$ (which depend on the uniform and $L^2$-norms of $f$ and $g$ and of finitely many of their derivatives) so that for all $t \geq 1$ we have
$$
\left\lvert \int_M f \circ h^{\tau} \cdot g \diff \mu^{\tau} - \mu^{\tau}(f) \mu^{\tau}(g) \right\rvert \leq B_\tau {\widetilde S}(f) {\widetilde S}(g) t^{-\alpha}.
$$
\end{theorem}

The proof actually provides an explicit bound for the exponent $\alpha$ above in terms of $\beta$ in Theorem \ref{thm:exp_mix}.
In particular, $\alpha$ can be taken independent of $\tau$, and is at least $\beta/8$. However, we do not claim that it is optimal, and we do not know whether the optimal exponent is actually independent of $\tau$.
Indeed, what are the optimal mixing rates is a question that is still open also in the case of time-changes of horocycle flows on compact surfaces. 
The best known bounds in that case, see \cite{forniulcigrai:timechanges}, coincide with the rate of equidistribution of sheared geodesic arcs. For the standard horocycle flow, it is known that the equidistribution of translates of geodesic segments is slower than the equidistribution of unstable horocycle arcs \cite{ravotti:horo}, and the latter matches the optimal mixing rates established by Ratner \cite{ratner2}. It is therefore possible that, in the case of smooth time-changes, a different approach based on shearing of curves transverse to the weak-stable leaves of the geodesic flows would provide sharper mixing estimates.


\section{Estimates on ergodic integrals}

Let $f \in H=L^2(M)$ and, for every $t \in \R$, define $I_tf$ to be the ergodic integral 
$$
I_tf(x) = \int_0^t f \circ h_r(x) \diff r.
$$
We are interested in the behaviour of $I_tf$ for large $t\geq 1$. When $M$ is a compact quotient of $\SL(2,\R)$ and $f$ is sufficiently smooth, the works of Burger \cite{bur}, Flaminio and Forni \cite{ffhoro}, and Bufetov and Forni \cite{bufo} provide sharp uniform bounds for $I_tf$. Already in the case where $M$ is non-compact, no uniform bound is possible: since the horocycle flow is not uniquely ergodic, the bounds on $I_tf(x)$ heavily depend on the starting point $x$, as it could lie on (or very close to) a periodic orbit, see \cite{ffhoro, Stro}. For unipotent flows which are not horospherical, no general pointwise estimate for ergodic integrals is known.

Estimates on deviations of ergodic averages are needed to carry out a \lq\lq mixing via shearing\rq\rq\ argument; however, it is not necessary to have control over the whole space. In this section, from the effective mixing result in Theorem \ref{thm:exp_mix}, we obtain $L^2$-bounds on $I_tf$, from which we deduce pointwise bounds on a set of large measure. The main estimate is the following.

\begin{proposition}\label{thm:bound_erg_int}
There exist $0<\gamma <1$ and ${\widetilde B} >0$ such that, for all $f \in H_0^\infty \cap L^\infty(M)$, the following holds. 
For every $T_0 \geq 1$, there exists a measurable set $E(f,T_0) \subset M$ with $\mu(E(f,T_0)) \leq T_0^{-\gamma}$ such that for all $t \geq T_0$, we have
$$
|I_tf(x)| \leq {\widetilde B} (S(f) + \|f\|_\infty) t^{1-\gamma} \text{\ \ \ for all\ } x \in M \setminus E(f,T_0).
$$ 
\end{proposition}

In order to obtain Proposition \ref{thm:bound_erg_int}, we first establish $L^2$-estimates, see, e.g., \cite[\S3]{ratneracta}.

\begin{lemma}\label{thm:L2estim}
There exist $0<\beta_0<1$ and $B'>0$ such that, for all $f \in H_0^\infty$ and all $t \geq 1$ we have
$$
\| I_tf\|_2^2 \leq B' S(f)^2 t^{2 - \beta_0}.
$$
\end{lemma}
\begin{proof}
Let $t\geq 1$ be fixed. Applying the Fubini-Tonelli Theorem several times, we have
\begin{equation*}
\begin{split}
\| I_tf\|_2^2 &= \int_M \int_0^t \int_0^t (f \circ h_r)(x) \cdot (f \circ h_s) (x)\diff r \diff s \diff \mu \\
&= \int_{[0,t]^2} \int_M (f\circ h_{r-s})(x) \cdot f(x) \diff \mu \diff s \diff r,
\end{split}
\end{equation*}
where the last equality follows from measure-invariance. Let $\beta$ be as in Theorem \ref{thm:exp_mix}, and define $\Delta = \{(t_1,t_2) \in [0,t]^2 : |t_1-t_2|\leq t^{\frac{1}{1+\beta}} \}$. From the equation above, we get
\begin{equation*}
\begin{split}
\| I_tf\|_2^2 &\leq \|f\|_2^2 \Leb(\Delta) +  \int_{[0,t]^2 \setminus \Delta} \int_M (f\circ h_{r-s})(x) \cdot f(x) \diff \mu \diff s \diff r\\
&\leq 4\|f\|_2^2 t^{1+\frac{1}{1+\beta}} + t^2 \sup \left\{ \left\lvert  \langle f \circ h_{r-s}, f \rangle  \right\rvert : t^{\frac{1}{1+\beta}} < |r-s| \leq t\right\}.
\end{split}
\end{equation*}
By Theorem \ref{thm:exp_mix}, the claim follows by choosing $B' = 4  +B$ and $\beta_0=\beta/(1+\beta)$.
\end{proof}

We now prove Proposition \ref{thm:bound_erg_int}. Using the $L^2$-bounds in Lemma \ref{thm:L2estim} and Chebyshev's Inequality, for any given time $t$, it is easy to deduce a pointwise bound for $I_tf(x)$ for all $x$ in a set of large measure (which depends on $t$). In order to obtain pointwise estimates that apply to all times greater than a given $t_0$, we use a simple approximation argument.

\begin{proof}[{Proof of Proposition \ref{thm:bound_erg_int}}]
Let $\gamma = \frac{\beta_0}{4}$. Fix $T_0\geq 1$ and define $N_0= \lfloor T_0^{\gamma } \rfloor + 2$.
Consider the sequence $k_n = n^{\gamma^{-1}}$ and define
$$
E_n = \left\{ x \in M: \left\lvert I_{k_n}(x) \right\rvert \geq \sqrt{ B' }S(f) k_n^{1-\gamma }\right\}, \text{\ \ \ and\ \ \ } E(f, T_0) = \bigcup_{n \geq N_0} E_n.
$$
By Chebyshev's Inequality and Lemma \ref{thm:L2estim},
$$
\mu ( E(f, T_0)) \leq \sum_{n = N_0}^{\infty} \mu(E_n) \leq \sum_{n = N_0}^{\infty} (B' S(f)^2)^{-1} k_n^{2\gamma -2}  \|I_{k_n}f\|_2^2 \leq \sum_{n = N_0}^{\infty} n^{-2} \leq (N_0-1)^{-1} \leq T_0^{-\gamma }.
$$

Let now $t \geq T_0$, and consider $n=\lfloor t^{\gamma } \rfloor + 2 \geq N_0$; in particular $k_n = n^{\gamma^{-1}} > t$. Note moreover that there exists a constant $B_0$ such that
$$
|k_n - t| \leq (t^{\gamma } + 2)^{\gamma^{-1}} - t \leq B_0 t^{1-\gamma}.
$$
Thus, for all $x \in M$, we have
$$
\left\lvert I_{k_n}f(x) - I_tf(x)\right\rvert \leq \left\lvert \int_t^{k_n} f \circ h_s(x) \diff s\right\rvert \leq B_0 \|f\|_\infty t^{1-\gamma}.
$$
If $x \in M \setminus E(f, T_0)$, in particular $x \notin E_n$, then $|I_{k_n}f(x)|<\sqrt{ B'} S(f) k_n^{1-\gamma}$. It follows that
\begin{equation*}
\begin{split}
|I_tf(x)| &\leq |I_{k_n}f(x)| + B_0 \|f\|_\infty t^{1-\gamma} \leq \sqrt{ B'} S(f) (t+B_0 t^{1-\gamma})^{1-\gamma}+ B_0 \|f\|_\infty t^{1-\gamma} \\
&\leq {\widetilde B} (S(f) + \|f\|_\infty)t^{1-\gamma},
\end{split}
\end{equation*}
where ${\widetilde B} = \sqrt{ B'} (B_0+1)^{1-\gamma} + B_0$.
\end{proof}

\section{Shear and distortion of pushed geodesic segments}

For the sake of notation, for every point $x \in M$, we will denote $x_r = \phi^X_r(x)$, where $\{\phi^X_t\}_{t \in \R}$ is the \lq\lq geodesic flow\rq\rq\ given by the Jacobson-Morozov Theorem.
We are interested in the push-forward of short geodesic segments under the action of the time-change.
We notice that, by the commutation relation \eqref{eq:commutation}, we have
$$
h^\tau_t (x_r) = h_{u(x_r,t)} \circ \phi^X_r(x) = \phi^X_r \circ h_{e^ru(x_r,t)} (x).
$$
In this section, we provide some estimates on the first and second derivative of $e^ru(x_r,t)$, which control the shear and the distortion of the geodesic arc under the action of $h^\tau_t$, and are analogous to the results in \cite[\S3]{forniulcigrai:timechanges}. We will use this estimates in the next section.

\begin{lemma}\label{thm:change_variab}
Define
$$
v(r,x,t) = t- \int_0^{u(x_r,t)} X\tau \circ h_{s} (x_r) \diff s.
$$
We have
$$
\frac{\partial}{\partial r} \left(e^r u(x_r,t) \right) = \frac{e^r v(r,x,t)}{\tau \circ h^{\tau}_t (x_r)}. 
$$
\end{lemma}
\begin{proof}
Since
\begin{equation}\label{eq:1st_line}
\frac{\partial}{\partial r} \left(e^r u(x_r,t) \right) = e^r \left( u(x_r,t)+ \frac{\partial}{\partial r} u(x_r,t)\right),
\end{equation}
we focus on the derivative of $u(x_r,t)$ with respect to $r$. Differentiating the equality
$$
t= \int_0^{u(x_r,t)} \tau \circ h_{s} (x_r) \diff s,
$$
we get
$$
0 = \left( \frac{\partial}{\partial r} u(x_r,t)\right) \tau \circ h_{u(x_r,t)} (x_r) + \int_0^{u(x_r,t)} \frac{\partial}{\partial r} \left( \tau \circ h_{s} (x_r) \right) \diff s,
$$
and therefore
\begin{equation}\label{eq:2nd_line}
\frac{\partial}{\partial r} u(x_r,t) = -(\tau \circ h^{\tau}_t (x_r))^{-1} \left( \int_0^{u(x_r,t)} \frac{\partial}{\partial r} \left( \tau \circ h_{s} \circ \phi^X_r(x) \right) \diff s \right).
\end{equation}
It is easy to check that $Dh_s(X) = X+sU$. This gives us
\begin{equation*}
\begin{split}
\frac{\partial}{\partial r} \left( \tau \circ h_{s} \circ \phi^X_r(x) \right) &= (X + sU)\tau \circ h_{s} \circ \phi^X_r(x) \\
&= X\tau \circ h_{s} \circ \phi^X_r(x)- \tau \circ h_{s} \circ \phi^X_r(x) + \frac{\partial}{\partial s}\left( s \tau \circ h_{s} \circ \phi^X_r(x) \right). 
\end{split}
\end{equation*}
We substitute the expression above into \eqref{eq:2nd_line} and we conclude
\begin{equation}\label{eq:du}
\frac{\partial}{\partial r} u(x_r,t) = -(\tau \circ h^{\tau}_t (x_r))^{-1} \left( \int_0^{u(x_r,t)} X\tau \circ h_{s} (x_r) \diff s - t\right) - u(x_r,t).
\end{equation}
Combining \eqref{eq:du} and \eqref{eq:1st_line} completes the proof.
\end{proof}

We now show that the function $v(x,r,t)$ is of order $t$ on a set of large measure, and its derivative is no larger than $t$, which will imply that the distortion is of order $1/t$ on a set of large measure.

\begin{lemma}[Control on the shear]\label{thm:shear}
There exist $0<\gamma <1$ and a constant $C_v\geq 1$ such that for every $t_0 \geq 1$ there exists a measurable set $E_v(t_0) \subset M$ with $\mu(E_v(t_0) ) \leq C_v t_0^{-\gamma}$ such that for all $t\geq t_0$ we have
$$
|v(x,r,t)-t| \leq C_v (rt+t^{1-\gamma}) \text{\ \ \ for all\ } x \in M \setminus E_v(t_0) \text{\ and all\ } r \in [0,1].
$$
\end{lemma}
\begin{proof}
Using the commutation relations between $h_s$ and $\phi^X_r$ in \eqref{eq:commutation}, we have
$$
\left\lvert v(x,r,t)-t \right\rvert =  \left\lvert \int_0^{u(x_r,t)} X\tau \circ \phi^X_r \circ h_{e^rs}(x) \diff s  \right\rvert \leq \left\lvert \int_0^{u(x_r,t)} X\tau \circ h_{e^rs}(x) \diff s  \right\rvert + m_\tau^2 r t,
$$
where we used the fact that $\|X\tau \circ \phi^X_r - X\tau \|_\infty \leq \|X^2\tau\|_\infty r \leq m_\tau r$, together with Lemma \ref{thm:ulessm}. Hence it remains to bound the first summand:
$$
\left\lvert \int_0^{u(x_r,t)} X\tau \circ h_{e^rs}(x) \diff s  \right\rvert = \left\lvert \frac{1}{e^r} \int_0^{e^r u(x_r,t)} X\tau \circ h_{s}(x) \diff s  \right\rvert
$$

Let $\gamma$ and ${\widetilde B}$ be given by Proposition \ref{thm:bound_erg_int}.
Fix $t_0 \geq 1$, and let $E_v(t_0):= E(X\tau, m_\tau^{-1}t_0)$ be the set given by Proposition \ref{thm:bound_erg_int} with $T_0 =m_\tau^{-1} t_0$; in particular $\mu(E_v(t_0)) \leq m_\tau^\gamma t_0^{-\gamma}$. Since 
$$
|e^r u(x_r,t)| \geq m_\tau^{-1}t \geq m_\tau^{-1}t_0 = T_0, 
$$
for all $x \notin E_v(t_0)$, we have
$$
\left\lvert \frac{1}{e^r}\int_0^{e^r u(x_r,t)} X\tau \circ h_{s}(x) \diff s \right\rvert \leq {\widetilde B}(S(X\tau) + m_\tau)(e^r u(x_r,t))^{1-\gamma} \leq {\widetilde B}(em_\tau)^{1-\gamma}(S(X\tau) + m_\tau) t^{1-\gamma}.
$$
The proof is therefore complete.
\end{proof}

\begin{lemma}[Control on the distortion]\label{thm:distortion}
There exists a constant $C_\tau \geq 1$ such that 
$$
\left\lvert \frac{\partial}{\partial r}v(x,t,r)\right\rvert \leq C_\tau t,
$$ 
for all $t,r \in \R$.
\end{lemma}
\begin{proof}
From the definition ov $v$ in Lemma \ref{thm:change_variab}, we directly compute
\begin{equation}\label{eq:lemma_dist}
\frac{\partial}{\partial r}v(x,t,r) = \left( \frac{\partial}{\partial r} u(x_r,t) \right) X\tau \circ h^{\tau}_t(x_r) + \int_0^{u(x_r,t)} \frac{\partial}{\partial r} X\tau \circ h_{s} \circ \phi^X_r(x) \diff s,
\end{equation}
and we estimate the two summands separately.
As in the proof of Lemma \ref{thm:shear}, we have
\begin{equation*}
\begin{split}
& \int_0^{u(x_r,t)} \frac{\partial}{\partial r} \left( X\tau \circ h_{s} \circ \phi^X_r(x) \right) \diff s=  \int_0^{u(x_r,t)} (X + sU)X\tau \circ h_{s} \circ \phi^X_r(x) \diff s \\
&\quad =  \int_0^{u(x_r,t)}X^2\tau \circ h_{s} \circ \phi^X_r(x)  \diff s  -  \int_0^{u(x_r,t)}X\tau \circ h_{s} \circ \phi^X_r(x)\diff s +  \int_0^{u(x_r,t)} \frac{\partial}{\partial s}\left( s X\tau \circ h_{s} \circ \phi^X_r(x) \right) \diff s\\
&\quad = \int_0^{u(x_r,t)}X^2\tau \circ h_{s} \circ \phi^X_r(x)\diff s -  \int_0^{u(x_r,t)}X\tau \circ h_{s} \circ \phi^X_r(x) \diff s +  u(x_r,t) \left(  X\tau \circ h^{\tau}_{t} (x_r) \right).
\end{split}
\end{equation*}
Therefore, by Lemma \ref{thm:ulessm}, we can bound the second summand in \eqref{eq:lemma_dist} by 
\begin{equation}\label{eq:summ2}
\left\lvert \int_0^{u(x_r,t)} \frac{\partial}{\partial r} \left( X\tau \circ h_{s} \circ \phi^X_r(x) \right) \diff s\right\rvert \leq (\|X^2\tau\|_\infty + 2\|X\tau\|_\infty) | u(x_r,t) | \leq 3m_\tau^2 t.
\end{equation}
By \eqref{eq:du}, we have
\begin{equation}\label{eq:summ1}
\left\lvert \frac{\partial}{\partial r} u(x_r,t)\right\rvert \leq m_\tau( \|X\tau\|_\infty m_\tau t + t) + m_\tau t \leq 3m_\tau^3 t.
\end{equation}
From \eqref{eq:summ2}, \eqref{eq:summ1}, and the fact that $\|X\tau\|_\infty \leq m_\tau$ by assumption, the claim follows.
\end{proof}

\section{Proof of the main result}

\subsection{Mixing via shearing}

For any $f \in L_0^2(M, \mu^{\tau})$, $t \in \R$ and $s >0$, define
$$
A_{t,s}f(x) = \int_0^s f \circ h^\tau_t \circ \phi^X_r(x) \diff r.
$$

\begin{lemma}\label{thm:mix_via_s}
Let $f \in L_0^2(M, \mu^{\tau})$ be bounded, and let $g \in L^2(M, \mu^{\tau})$ be such that $Xg \in L^2(M, \mu^{\tau})$. For every $t \in \R$ and for all $\sigma >0$, we have
$$
\left\lvert \int_M f \circ h^{\tau} \cdot g \diff \mu^{\tau} \right\rvert \leq \left( \|\tau g \|_2 +\sigma \|X(\tau g) \|_2 \right)\frac{ 1 }{\sigma}  \sup_{s \in [0,\sigma]} \|A_{t,s} f\|_2.
$$
\end{lemma}
\begin{proof}
By definition of $\mu^\tau$ and the invariance properties of $\mu$, we rewrite
$$
\int_M f \circ h^{\tau} \cdot g \diff \mu^{\tau} = \int_M f \circ h^{\tau} \cdot (\tau g) \diff \mu = \langle f \circ h^{\tau} , (\tau g)\rangle = \langle f \circ h^{\tau} \circ \phi^X_s, (\tau g) \circ \phi^X_s\rangle,
$$
for all $s \in \R$. Therefore, given $\sigma >0$, 
$$
\int_M f \circ h^{\tau} \cdot g \diff \mu^{\tau} = \frac{1}{\sigma} \int_0^\sigma \langle f \circ h^{\tau} \circ \phi^X_s, (\tau g) \circ \phi^X_s\rangle \diff s.
$$
After integrating by parts, we obtain
$$
\int_M f \circ h^{\tau} \cdot g \diff \mu^{\tau} = \frac{1}{\sigma} \langle \int_0^\sigma f \circ h^{\tau} \circ \phi^X_s \diff s, (\tau g) \circ \phi^X_\sigma \rangle -  \frac{1}{\sigma}\int_0^\sigma \langle \int_0^s f \circ h^{\tau} \circ \phi^X_r \diff r, X(\tau g) \circ \phi^X_s \rangle \diff s,
$$
and, by Cauchy-Schwarz inequality,
$$
\left\lvert \int_M f \circ h^{\tau} \cdot g \diff \mu^{\tau} \right\rvert \leq \frac{ \|\tau g\|_2 }{\sigma} \|A_{t,\sigma}f \|_2+ \|X(\tau g)\|_2 \sup_{s \in [0,\sigma]} \|A_{t,s}f\|_2.
$$
\end{proof}


\subsection{Proof of Theorem \ref{thm:main}}

We now prove Theorem \ref{thm:main}. First of all, notice that $f \in L^2(M, \mu^{\tau})$ if and only if $f \in H=L^2(M, \mu)$, and $f \in L_0^2(M, \mu^{\tau})$ implies that $\tau f \in L_0^2(M, \mu)$.

Let $f,g \in \mathscr{C}^{\infty}_c(M)$; in particular $f, g \in H^\infty $ and $f, Xf \in L^\infty(M)$. Let $t \geq 1$ be fixed, and define
$$
t_0 = \sqrt{t}, \text{\ \ \ } \sigma = t^{-\gamma/2}, \text{\ \ \ and\ \ \ } E(t) = E_v(t_0) \cup h^{\tau}_{-t}(E(\tau f,t_0)),
$$
where $E_v(t_0)$, and $E(\tau f,t_0)$ are the sets given by Lemma \ref{thm:shear} and Proposition \ref{thm:bound_erg_int} respectively. 
By definition of $\mu^\tau$, we have
$$
\mu(h^{\tau}_{-t}(E(\tau f,t_0))) \leq m_\tau \mu^{\tau}(h^{\tau}_{-t}(E(\tau f,t_0))) = m_\tau \mu^{\tau}(E(\tau f,t_0)) \leq m_\tau^2\mu (E(\tau f,t_0)) \leq m_\tau^2 t^{-\gamma/2},
$$
in particular, we can bound
\begin{equation}\label{eq:L2norm_A}
\begin{split}
\|A_{t,s}f\|_2 &\leq \sup_{x \in M\setminus E(t)} |A_{t,s}f(x)| + \|f\|_{\infty} \sigma \mu(E(t)) \\
& \leq \sup_{x \in M\setminus E(t)} |A_{t,s}f(x)| +  (C_v+m_\tau^2)\|f\|_{\infty} \sigma t^{-\gamma/2}.
\end{split}
\end{equation}
We now estimate $|A_{t,s}f(x)|$ for $x \in M\setminus E(t)$.
Notice that for all such points, by Lemma \ref{thm:shear}, we have
$$
\left\lvert v(x,r, t) - t \right\rvert \leq C_v t^{1-\gamma/2}, \text{\ \ \ for all\ }r \in [0,\sigma].
$$
Up to increasing the constant at the end of the proof, we will assume that $t \geq (2C_v)^{2/\gamma}$, so that 
\begin{equation}\label{eq:v_geq_t}
\left\lvert v(x,r, t) \right\rvert \geq \frac{t}{2}, \text{\ \ \ for all\ }r \in [0,\sigma].
\end{equation}

\begin{lemma}\label{thm:d_1_on_v}
There exists a constant $C_\tau' \geq 1$ such that for all $x \in M \setminus E(t)$ and all $s \in [0, \sigma]$, we have
$$
\left\lvert \frac{\partial}{\partial r} \frac{1}{e^r v(x,r,t)}\right\rvert \leq \frac{C_\tau'}{t}.
$$
\end{lemma}
\begin{proof}
The claim follows immediately from Lemma \ref{thm:distortion} and \eqref{eq:v_geq_t}, since
$$
\left\lvert \frac{\partial}{\partial r} \frac{1}{e^r v(x,r,t)}\right\rvert \leq \left\lvert \frac{1}{e^r v(x,r,t)}\right\rvert + \left\lvert \frac{\frac{\partial}{\partial r}v(x,r,t)}{e^{2r} v^2(x,r,t)}\right\rvert \leq \frac{4(1+C_\tau)}{t}.
$$
\end{proof}

\begin{proposition}\label{thm:bound_on_A}
There exists a constant $B_\tau' \geq 1$ depending on $\tau$ only such that for all $x \in M \setminus E(t)$ and all $s \in [0, \sigma]$, we have
$$
|A_{t,s}f(x)| \leq B_\tau' ({\widehat S}(f) + \|f\|_\infty+ \|Xf\|_\infty) t^{-\gamma},
$$
where ${\widehat S}(f)$ is a Sobolev norm of $f$. 
\end{proposition}
\begin{proof}
By the commutation relations between $h_t$ and $\phi^X_r$, it follows that
\begin{equation*}
\begin{split}
A_{t,s}f(x) &= \int_0^s \frac{\tau f}{\tau} \circ h^{\tau}_t \circ \phi^X_r(x) \diff r =  \int_0^s \frac{(\tau f)\circ \phi^X_r \circ h_{e^ru(x_r,t)}(x) }{\tau \circ h^{\tau}_t (x_r)}\diff r \\
&= \int_0^s \frac{(\tau f) \circ h_{e^ru(x_r,t)}(x) }{\tau \circ h^{\tau}_t (x_r)}\diff r + \int_0^s \frac{[(\tau f)\circ \phi^X_r - (\tau f)] \circ h_{e^ru(x_r,t)}(x) }{\tau \circ h^{\tau}_t (x_r)}\diff r.
\end{split}
\end{equation*}
Since $\|(\tau f) \circ \phi^X_r - (\tau f)\|_\infty \leq \sigma \|X(\tau f)\|_\infty$, it follows that
$$
|A_{t,s}f(x)| \leq \left\lvert \int_0^s \frac{(\tau f) \circ h_{e^ru(x_r,t)}(x) }{\tau \circ h^{\tau}_t \circ \phi^X_r(x)}\diff r \right\rvert + m_\tau \sigma^2 \|X(\tau f)\|_{\infty}.
$$
By the choice of $\sigma = t^{-\gamma/2}$, the second summand satisfies the desired bound. It remains to estimate the first term in the right hand-side above.
We multiply and divide by $e^r v(x,r,t)$ so that, using Lemma \ref{thm:change_variab}, by an integration by parts and a change of variable, we obtain
\begin{equation*}
\begin{split}
\int_0^s \frac{(\tau f) \circ h_{e^ru(x_r,t)}(x) }{\tau \circ h^{\tau}_t \circ \phi^X_r(x)}\diff r =& \int_0^s (\tau f) \circ h_{e^ru(x_r,t)}(x) \frac{\partial}{\partial r} \left(e^r u(x_r,t) \right) \frac{1}{e^rv(x,r,t)} \diff r \\
=& \frac{1}{e^sv(x,s,t)} \left(\int_{u(x,t)}^{e^su(x_s,t)} (\tau f) \circ h_r(x) \diff r\right) \\
& - \int_0^s \left( \frac{\partial}{\partial r} \frac{1}{e^rv(x,r,t)} \right) \left( \int_{u(x,t)}^{e^ru(x_r,t)} (\tau f) \circ h_{\ell}(x) \diff \ell \right) \diff r.
\end{split}
\end{equation*}
Therefore, by Lemma \ref{thm:d_1_on_v} and \eqref{eq:v_geq_t},
\begin{equation*}
\begin{split}
\left\lvert \int_0^s \frac{(\tau f) \circ h_{e^ru(x_r,t)}(x) }{\tau \circ h^{\tau}_t \circ \phi^X_r(x)}\diff r \right\rvert \leq & \frac{2}{t} \left\lvert \int_{0}^{e^su(x_s,t) - u(x,t)} (\tau f) \circ h_r(h^{\tau}_t(x)) \diff r\right\rvert \\
&+ \sigma \frac{C_\tau'}{t} \sup_{r\in [0,\sigma]} \left\lvert \int_{0}^{e^ru(x_r,t) - u(x,t)} (\tau f) \circ h_\ell (h^{\tau}_t(x)) \diff \ell \right\rvert.
\end{split}
\end{equation*}
We consider two possible cases: if $r \in [0,\sigma]$ is such that $|e^ru(x_r,t) - u(x,t)|\leq t_0 =\sqrt{t}$, then, clearly,
$$
\left\lvert \int_{0}^{e^ru(x_r,t) - u(x,t)} (\tau f) \circ h_\ell (h^{\tau}_t(x)) \diff \ell \right\rvert \leq \|\tau f\|_\infty  \sqrt{t}.
$$
Otherwise, if $|e^ru(x_r,t) - u(x,t)| > t_0$, then, since by assumption $h^{\tau}_t(x) \notin E(\tau f, t_0)$, we have
\begin{equation*}
\begin{split}
\left\lvert \int_{0}^{e^ru(x_r,t) - u(x,t)} (\tau f) \circ h_\ell (h^{\tau}_t(x)) \diff \ell \right\rvert &\leq {\widetilde B} (S(\tau f) + \|\tau f\|_\infty) |e^ru(x_r,t) - u(x,t)|^{1-\gamma} \\
&\leq {\widetilde B}m_\tau (e+1)  (S(\tau f) + \|\tau f\|_\infty) t^{1-\gamma}.
\end{split}
\end{equation*}
It is possible to find a Sobolev norm ${\widehat S}$ such that $S(\tau f) \leq {\widehat S}(f){\widehat S}(\tau)$ (see, e.g., \cite[\S2.2]{BEG}). Therefore, in either of the previous cases, we conclude
$$
|A_{t,s}f(x)| \leq B_\tau' ({\widehat S}(f) + \|f\|_\infty + \|Xf\|_\infty)t^{-\gamma},
$$
for a constant $B_\tau'$ which depends on $\tau$ only.
\end{proof}

\begin{proof}[{End of the proof of Theorem \ref{thm:main}}]
Since $f,g \in \mathscr{C}^{\infty}_c(M)$, then $f,g \in L^2(M, \mu^\tau)$. 
Without loss of generality, we can assume that $f$ has zero average with respect to $\mu^\tau$, namely $\mu^\tau(f) = \mu(\tau f) =0$. 
By Lemma \ref{thm:mix_via_s} and \eqref{eq:L2norm_A}, we have
\begin{equation*}
\begin{split}
\left\lvert \int_M f \circ h^{\tau} \cdot g \diff \mu^{\tau} \right\rvert &\leq \left( \|\tau g \|_2 + \|X(\tau g) \|_2 \right) \frac{ 1 }{\sigma} \sup_{s \in [0,\sigma]} \|A_{t,s} f\|_2 \\
&\leq \left( \|\tau g \|_2 + \|X(\tau g) \|_2 \right) \left( (C_v+m_\tau^2)\|f\|_{\infty} t^{-\gamma/2} + t^{\gamma/2} \sup_{x \in M\setminus E(t)} |A_{t,s}f(x)| \right).
\end{split}
\end{equation*}
By Proposition \ref{thm:bound_on_A}, we conclude
$$
\left\lvert \int_M f \circ h^{\tau} \cdot g \diff \mu^{\tau} \right\rvert \leq B_\tau \widetilde{S}(f) \widetilde{S}(g) t^{-\gamma/2},
$$
where $\widetilde{S}(\cdot ) = \widehat{S}(\cdot ) + \|\cdot\|_\infty + \|X( \cdot)\|_\infty$ and $B_\tau$ is a constant depending on $\tau$ only.
\end{proof}

\subsection*{Acknowledgements}
I would like to thank Mauro Artigiani, Giovanni Forni, and Corinna Ulcigrai for several useful discussions and for their comments on a previous version of the paper. This research was partially funded by the Australian Research Council.

\end{document}